\documentclass[12pt,a4paper,oneside]{amsart}
\usepackage{amsfonts, amsmath, amssymb, amsthm}
\usepackage[colorlinks=true,citecolor=blue]{hyperref}
\usepackage[margin=1.4in]{geometry}
\usepackage{txfonts}
\usepackage[T1]{fontenc}
\usepackage{bbm}
\usepackage{mathtools}
\usepackage{graphicx}
\usepackage{enumerate}
\usepackage{verbatim}
\usepackage{tikz}
\usepackage{algorithmic}

\usetikzlibrary{decorations,decorations.pathreplacing}

\begingroup
    \makeatletter
    \@for\theoremstyle:=definition,remark,plain\do{%
        \expandafter\g@addto@macro\csname th@\theoremstyle\endcsname{%
            \addtolength\thm@preskip\parskip
            }%
        }
\endgroup

\newtheorem{theorem}{Theorem}
\newtheorem*{theorem*}{Theorem}
\newtheorem{lemma}[theorem]{Lemma}

\newtheorem{conjecture}[theorem]{Conjecture}

\theoremstyle{definition}

\newtheorem*{remark*}{Remark}

\newcommand{\Q}{\mathbb{Q}}
\newcommand{\R}{\mathbb{R}}
\newcommand{\LL}{\mathcal{L}}
\newcommand{\s}{\mathbb{S}}

\newcommand{\pth}[1]{\left(#1\right)}

\usepackage{color}
\definecolor{blue}{rgb}{0,0,1}
\definecolor{red}{rgb}{1,0,0}
\definecolor{grey}{rgb}{0.6,0.6,0.6}
\definecolor{purple}{rgb}{0.6,0,1}

\begin{document} 

\title{The Topology of the set of line Transversals}
\author{Otfried~Cheong \and Xavier~Goaoc \and Andreas~F.~Holmsen}

\maketitle 

\begin{abstract}
  We prove that for any set $F$ of $n\ge 2$ pairwise disjoint open convex sets in $\R^3$, the connected components of the set of lines intersecting every member of $F$ are contractible. The  same result holds for directed lines.
\end{abstract}

\section{Introduction}

Let $F$ be a family 
of convex sets in $\R^d$. A {\em line transversal} to
$F$ is a line that intersects every member of $F$. Here is the main
result that we prove in this note:

\begin{theorem}\label{thm:toptrans}
  For any family $F$ of $n \ge 2$ pairwise disjoint open convex sets
  in $\mathbb{R}^3$, every connected component of line transversals to
  $F$ is acyclic.
\end{theorem}

\noindent
Here, the space of lines in $\R^3$ is endowed with its natural
topology (see Section~\ref{s:background}). Recall that a set is {\em
  acyclic} if it has trivial homology in dimension~$1$ and
higher. 
The statement of Theorem~\ref{thm:toptrans} remains true when the lines are directed. A natural question is whether the assumption that the sets be open can be replaced by a compactness condition; We do not know if 
this
is possible, but offer some comments in the final section.

\bigskip

Our proof goes through the analysis of the {\em cone of directions}
$T(F)$ of the family $F$ of convex sets, defined as the set of vectors
$x\in \s^2$ for which there exists a line transversal to $F$ in the
direction $x$. We first establish that $\s^2 \setminus T(F)$ is
path-connected (Lemma~\ref{boundaries}, and the cornerstone of our
proof), and it is a well-known fact that this implies that the connected components of $T(F)$ are contractible. With the Vietoris-Begle mapping
theorem, this implies Theorem~\ref{thm:toptrans}.

\subsection{Context and motivation}

Our motivation to study the topology of sets of line transversals to
convex sets originates in questions raised in combinatorial geometry,
more precisely in {\em geometric transversal theory}. We refer
to the chapters by Goodman, Pollack and Wenger~\cite{gpw-gtt-93} and
by Goodman and Pollack~\cite{gp-gtt-02} or the survey of Holmsen and
Wenger~\cite{holmsen2017helly} for an overview.

\bigskip

Helly's theorem asserts that for any finite family $F$ of convex sets
in $\R^d$, if every $d+1$ members of $F$ have a point in common, then
all members of $F$ have a point in common. A natural question is
whether a similar phenomenon holds for higher-dimensional
transversals, for instance:

\bigskip
\begin{quote}
  Is there a constant $H_d$ such that for any finite family of
  pairwise disjoint convex sets in $\R^d$, if every $H_d$ members of
  $F$ have a line transversals then $F$ has a line transversal.
\end{quote}

\noindent
The answer is negative in general, already for $d=2$, but
positive under further restrictions on $F$. In particular, the
influence of the geometry of the members of $F$ on the existence of
Helly-type theorems was extensively investigated, leading to many
specialized results and techniques on, for instance,
parallelotopes~\cite{s-tscpap-40,ghprs-ccfcts-06}, disjoint translates
of a convex set in the plane~\cite{g-ct-58,katch86,t-pgcct-89} or in
$\R^3$~\cite{hm-nhtst-04}, disjoint planar convex sets with an
ordering condition~\cite{h-uegt-57}, disjoint
balls~\cite{Danzer57,g-ctfs-60,cghp-hhtdus-08,bgp-ltdb-08}, etc.

\bigskip

A more unified picture emerges when one considers the above question
through a topological lens. Helly's theorem enjoys several topological
generalizations, see e.g. ~\cite{helly1930, deb1970, luis2014}, and some of the recent ones relate the
Helly-type properties of a family of sets to some measure of its
topological complexity~\cite{acyclic,hb17}. These results allow to
retrieve in a single stroke several of the special cases once the
influences of the geometry of $F$ on the homology of the set of line
transversals to~$F$ has been elucidated. In the plane, any connected
component of line transversals to disjoint convex sets is
contractible, so the geometry only limits the number of connected
components of line transversals; See the discussion
in~\cite[$\mathsection~7$]{acyclic}. Theorem~\ref{thm:toptrans}
establishes that the same happens in $\R^3$ 
for open sets.
\footnote{Interestingly,
  in several of the geometric settings where Helly-type theorems exist
  for line transversals, their study proceeds through the analysis of
  {\em geometric permutations}, which appear, in hindsight, as a
  combinatorial characterization of the connected components of line
  transversals.} We conjecture that the same holds in higher
dimension:

\begin{conjecture}
  For any family $F$ of $n \ge 2$ pairwise disjoint open
  convex sets in $\mathbb{R}^d$, every connected
  component of line transversals to $F$ is acyclic.
\end{conjecture}

\bigskip


As a final remark it should be pointed out that the study of geometric transversals is not limited to line transversals. More generally, a {\em $k$-transversal} to a family of convex sets in $\mathbb{R}^d$ is a $k$-dimensional affine flat that intersects every member of the family.
Most studied is the case of 
{\em hyperplane transversals}, that is,  when $k=d-1$, and some notable results are; Necessary and sufficient conditions for the existence of hyperplane transversals \cite{gp1988, pw1990, wen1990, aw1996}, a $(p,q)$-theorem for hyperplane transversals \cite{ak1995}, and results on the topology of the space of hyperplane transversals \cite{abmos2002}. For $1<k<d-1$, much less is known, but there are some intriguing results concerning the topology of the space of $k$-transversals for small families of convex sets \cite{bm2002, mk2011}.

\section{Background}
\label{s:background}

We denote by $\LL$ the space of lines in $\R^3$ equipped with its
usual topology; this topology can be defined for instance via a
parameterization of that space ({\em e.g.}  by Pl\"ucker coordinates)
or by recasting it as a quotient space of an open subset of $\R^3
\times \R^3$, via the map that sends a pair of distinct points to the
line that they span. We denote by $\LL^+$ the the space of directed
lines in $\R^3$, equipped with its usual topology (which can be
defined similarly).

\bigskip

The Vietoris-Begle mapping theorem asserts that in certain conditions,
some of the topological and homological properties of a topological
space $X$ are preserved by continuous projections $p:X \to Y$ with
contractible fibers\footnote{Meaning that for every $y \in p(X)$, the
  set $p^{-1}(y)$ is contractible.}. As discussed
in~\cite[Appendix]{acyclic}, there are several formulations of this
theorem, which vary in the assumption on the spaces $X$ and $Y$ and in
the specific properties that are preserved. Here we will use the
following special case of the Vietoris-Begle mapping theorem (which
follows from~\cite[Lemma~26~(1)]{acyclic}):

\begin{lemma}[Vietoris-Begle for lines in $\R^3$]\label{lem:VB}
  Let $\pi:\LL^+ \to \s^2$ denote the map that associates to a
  directed line its direction. Let $U \subseteq X$ be an open set. If
  $\pi^{-1}(v) \cap U$ is contractible for every $v \in \pi(U)$, then
  for every integer $i \ge 0$ we have $H_i(U,\Q) \simeq
  H_i(\pi(U),\Q)$.
\end{lemma}

\bigskip

We will also use the following classical fact (see e.g. \cite[Section 4.2]{Ahlfors}):


\begin{lemma}\label{lem:complement}
  Let $U$ be an open subset of $\s^2$. Every connected component of
  $U$ is contractible if and only if $\s^2 \setminus U$ is
  path-connected.
\end{lemma}

\section{Building paths between non-transversal directions}

Let $F$ be a finite family of convex sets in $\mathbb{R}^3$. The set
of {\em transversal directions} for $F$ is the the set of vectors
$x\in S^2$ for which there exists a line transversal to $F$ in the
direction $x$. For a vector $x \in \s^2$, let $p_x : \mathbb{R}^3 \to
x^\perp$ denote the orthogonal projection into the orthogonal
complement of $x$. The set of transversal directions for $F$ can also
be expressed as
\begin{equation}\label{eq:TF}
  T(F) = \{x\in \s^2 : \textstyle{\bigcap_{K\in F} p_x(K) \neq \emptyset}.\}
\end{equation}
Note that $T(F)$ is an antipodally symmetric subset of $S^2$. The set
of {\em non-transversal directions} for $F$ is the set $\s^2 \setminus
T(F)$. 

\bigskip

Since the members of $F$ are pairwise disjoint {and open}, we can
choose, for any two members of $F$, some plane that separates them
strictly. By translating these separating hyperplanes to the origin,
we get an arrangement of great circles on $\s^2$. We call the union of
these great circles, one per pair of members of $F$, a set of {\em
  separating directions} for $F$. Note that any set of separating
directions for $F$ is a path-connected subset of $N(F)$. The following
is our main technical ingredient.

\begin{lemma} \label{boundaries}
  Let $F$ be a finite family of {at least two} pairwise disjoint
  open convex sets in $\mathbb{R}^3$. For any set $Y$ of separating
  directions for $F$ and any boundary point $v$ of $N(F)$, there
  exists a path $\gamma : [0,1] \to N(F)$ such that $\gamma(0) = v$
  and $\gamma(1) \in Y$.
\end{lemma}

\begin{proof}
Consider a point $v$ on the boundary of $N(F)$, that is, $v\in N(F)
\cap \overline{T(F)}$.  (Here $\overline{X}$ denotes the {\em closure}
of $X$.) {The formulation of $T(F)$ given in
  Equation~\eqref{eq:TF} makes it clear that when the members of $F$
  are open, so is $T(F)$. It follows that $N(F)$ is closed, and that }
we have
\[{\textstyle{\bigcap_{K\in F}}}p_v(K)
= \emptyset \neq {\textstyle{\bigcap_{K\in F}}}p_v(\overline{K}).\]
Indeed, the equality on the left must hold or else $v \notin N(F)$, and the
inequality on the right must hold or else $v$ is in the interior of
$N(F)$. It follows that ${\textstyle{\bigcap_{K\in
      F}}}p_v(\overline{K})$ has empty interior (as a subset of
$\mathbb{R}^2$). By a simple application of Helly's theorem there
exists a subfamily $G\subset F$ with $|G|\leq 3$ such that
$\bigcap_{K\in G} p_v(\overline{K})$ has empty interior. Since the
members of $F$ have nonempty interiors, $G$ must consist of two or
three members. These cases will be treated separately.

\begin{figure}[htbp]
    \centering
    \includegraphics{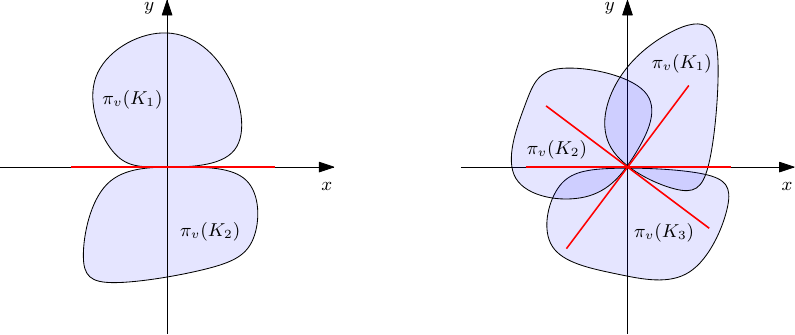}
    \caption{Caption}
    \label{fig:boundary}
\end{figure}

\medskip

Suppose $G = \{K_1,K_2\}$. (See \textsc{Figure} \ref{fig:boundary},
left.)  In this case $p_v(K_1)$ and $p_v(K_2)$ are strictly separated
by a line.  Let $n$ be the unit normal vector of a plane among the
separating directions for $F$ such that $x\cdot n < 0$ for every $x\in
K_1$ and $x\cdot n > 0$ for every $x \in K_2$.  (The vector $n$ must
exists since {the separating set $Y$ contains a great circle from a strictly separating plane for} every pair of sets.) Note that if
$v \cdot n = 0$, then $v\in Y$, and we are done by defining $\gamma$
to be constant path.  It remains to consider the situation when
$n\cdot v \neq 0$. Without loss of generality we may assume that, as depicted in Figure~\ref{fig:boundary} left,
\begin{itemize}
    \item $v=(0,0,1)$,
    \item $v\cdot n > 0$,
    \item $K_1$ lies in the open halfspace $y > 0$, and
    \item $K_2$ lies in the open halfspace $y < 0$.
\end{itemize}

For $\theta \in [0,\pi]$ define the vector $v(\theta)= (0,\sin \theta, \cos\theta)$, and 
let $\alpha \in (0,\pi)$ be the unique value such that $v(\alpha) \cdot n = 0$. 
This means that $v(\alpha) \in Y$, and we claim that for any $0 < \theta < \alpha$ we have $v(\theta)\in N(F)$. Setting $\gamma(t) = v(\alpha t)$ gives us the desired path. 

To see why $v(\theta)\in N(F)$, note first  that for $0 < \theta < \alpha$ we have $v(\theta)\cdot n >0$. This implies that if a line in the direction $v(\theta)$ intersects $K_1$ and $K_2$, then it must intersect $K_1$ before it intersects $K_2$. On the other hand, for any $0< \theta < \pi$, a line in the direction $v(\theta)$ must intersect the open halfspace $y<0$ before it intersects the open halfspace $y>0$. But this implies that such a line would have to intersect $K_2$ before it could intersect $K_1$. It follows that for $0 < \theta < \alpha$ the vector $v(\theta)$ is a non-transversal direction for $F$. 

\bigskip

Now suppose $G = \{ K_1, K_2, K_3 \}$. (See \textsc{Figure} \ref{fig:boundary}, right.) In this case, each $p_v(K_i)$ is contained in an open halfplane, such that the closures of these three halfplanes intersect in a single point. Moreover, we may assume that the pairwise intersections of the $p_v(K_i)$ are nonempty or else we may apply the  previous case analysis.

As before, let $n$ be the unit normal vector of a plane among the separating directions for $F$ such that $x\cdot n < 0$ for every $x\in K_1$ and $x\cdot n > 0$ for every $x \in K_2$. If $v \cdot n = 0$, then $v\in Y$ and we are done by defining $\gamma$ to be constant path, so we consider the situation when $n\cdot v \neq 0$. Without loss of generality we may assume that, as depicted in Figure~\ref{fig:boundary} right,
\begin{itemize}
    \item $v = (0,0,1)$, 
    \item $v\cdot n > 0$,
    \item $K_1$ lies in the open halfspace $y > ax$ for some $a<0$,
    \item $K_2$ lies in the open halfspace $y > bx$ for some $b>0$, and
    \item $K_3$ lies in the open halfspace $y< 0$.
\end{itemize}

For $\theta \in [0,\pi]$ define the vector $u(\theta) = (\sin\theta, 0, \cos\theta)$, and let $\beta \in (0,\pi)$ be the unique value such that $u(\beta) \cdot n = 0$. This means that $u(\beta)\in Y$, and we will show that for any $0 < \theta < \beta$ we have $u(\theta) \in N(F)$, which will complete the proof. 

As in the previous case, we note that for $0< \theta < \beta$ we have $u(\theta) \cdot n >0$. Therefore, if a line in the direction $u(\theta)$ intersects $K_1$ and $K_2$, then it must intersect $K_1$ before it intersects $K_2$. On the other hand, consider a line $\ell$ in the direction $u(\theta)$ for any $0< \theta< \pi$. If $\ell$ intersects $K_3$, then $\ell$ is contained in the open halfspace $y< 0$, since the $y$-coordinate of $u(\theta)$ equals $0$. If $\ell$ also intersects $K_1$ and $K_2$, then $\ell$ meets $K_2$ in the intersection of the open halfspaces $y<0$ and $x<0$, while $\ell$ meets $K_1$ in the intersection of the open half spaces $y<0$ and $x>0$. Since $u(\theta) \cdot (1,0,0) = \sin \theta  >0$, the line $\ell$ must intersect $K_2$ before it intersects $K_1$. It follows that for $0 < \theta < \beta$ the vector $u(\theta)$ is a non-transversal direction for $F$, and the proof of the claim is complete. 
\end{proof}

\section{Wrapping up}

We now use our path-building Lemma \ref{boundaries} to analyze sets of
directions and sets of line transversals.

\begin{lemma}\label{lem:open}
  Let $F$ be a finite family of at least two pairwise disjoint
  open convex sets in $\mathbb{R}^3$. The set $N(F) = \s^2 \setminus
  T(F)$ of non-transversal directions for $F$ is a path-connected
  subset of~$\s^2$.
\end{lemma}
\begin{proof}
  Let $x$ be a vector in $N(F)$. Starting at $x$, we move along an
  arbitrary geodesic of $\s^2$ until we either meet the boundary of
  $N(F)$ or we meet $Y$, stopping at whichever happens first. If we
  meet $Y$ first, then we are done since $Y$ is a path-connected
  subset of $N(F)$. If we meet the boundary of $N(F)$ first, then we
  change direction and move along the path given by Lemma
  \ref{boundaries}, and again we reach $Y$ by a path contained within
  $N(F)$.
\end{proof}

\noindent
Our main result now easily follows from Lemma~\ref{lem:open} by
arguments spelled out in~\cite[Lemma~24]{acyclic}, which we summarize
here for completeness.

\begin{proof}[Proof of Theorem~\ref{thm:toptrans}]
  Let $F$ be a finite family of at least two pairwise disjoint
  convex sets in $\mathbb{R}^3$. We first prove that the conected
  components of line transversals to $F$ are acyclic when (i) members
  of $F$ are open, and (ii) lines are directed.

  \bigskip
  
  First, let us remark that Lemmas~\ref{lem:open}
  and~\ref{lem:complement} imply that the connected components of
  $T(F)$ are contractible and therefore acyclic.

  \bigskip

  Now, let $\pi: \LL^+ \to \s^2$ denote the map that sends every
  directed line to its direction vector. Let $\LL^+(F)$ denote the set
  of directed line transversals to $F$. For any direction $v \in
  T(F)$, let $\LL_v^+(F)$ denote the set of directed line transversals
  to $F$ with direction $v$. The fact that $\LL_v^+$ is homeomorphic
  to $\bigcap_{K\in F} p_v(K)$ has two consequences: (i) for every $v
  \in T(F)$, the fiber $\pi^{-1}(v) \cap \LL^+(F)$ is contractible,
  and (ii) for every connected component $U$ of $\LL^+(F)$, $\pi(U)$
  is a connected component of $T(F)$. With Lemma~\ref{lem:VB}, the
  acyclicity of the connected components of $T(F)$ therefore implies
  the acyclicity of the connected components of $\LL^+$.
  
  \bigskip

   To go from directed to non-directed lines, let us fix a plane $\Pi$ separating two members of $F$ (which ones does not matter) and a vector $\vec{n}$ normal to $\Pi$. Let $\LL'(F)$ denote the set of directed line transversals to $F$ that make a positive dot product with $\vec{n}$. The map that associates each line in $\LL(F)$ to its orientation making a positive dot-product with $\vec{n}$ induces a homeomorphism between $\LL(F)$ and $\LL'(F)$. Since no line in $\LL^+(F)$ has a direction parallel to $\Pi$, every connected component of $\LL'(F)$ is a connected components of $\LL^+(F)$. Hence, every connected component of $\LL(F)$ is homeomorphic to a connected component of $\LL^+(F)$ and is also acyclic.
\end{proof}

\section{Concluding remarks}

Let us conclude with a construction that shows that for every compact subset $C \subset \R$, there exists a family $F$ of four pairwise disjoint compact convex sets in~$\R^3$ with $T(F)$ homeomorphic to $C$. This highlights some of the issues that arise when trying to replace the openness condition by a compactness assumption in Theorem~\ref{thm:toptrans}.

\bigskip

Let $\Sigma$ denote the hyperbolic paraboloid defined by the equation~$z = xy$. This surface is {\em doubly ruled}: every point $(\alpha,\beta,\alpha\beta)$ is on two lines contained in $\Sigma$, namely $\lambda_\alpha = \Sigma \cap \{x=\alpha\}$ and $\ell_\beta = \Sigma \cap \{y=\beta\}$. 
For $i=1,2,3$, let $S_i$ denote the closed line segment supported by $\lambda_i$ and bounded by the planes $y=1$ and $y=2$. The set of lines transversals to $\{S_1, S_2, S_3\}$ is exactly $\{\ell_b \colon b \in [1,2]\}$.

\bigskip

Now fix an arbitrary compact subset $C \subseteq [1,2]$ and let $\hat{C} = \left\{\pth{\frac1{c-4},c,\frac{c}{c-4}} \colon c \in C\right\}$. Every point of $\hat C$ belongs to the hyperbola with equation $x(y-4)=1$, which is the intersection of $\Sigma$ and the plane~$z = 4y + 1$.  Let $S_4$ denote the convex hull of $\hat C$. The set $S_4$ is compact and convex. Moreover, since $S_4$ is contained in the plane $z = 4y + 1$, we can see that $S_4 \cap \Sigma = \hat C$. It follows that
\[ T(\{S_1,S_2,S_3,S_4\}) = \{\ell_b \colon b \in C\} \simeq C,\]
as announced. Let us stress that using ideas showcased in~\cite{hm-nhtst-04} we can make each~$S_i$ into a convex {\em body}, with non-empty interior, without changing the set of line transversals of the family.

\bigskip

Consider this construction for $C$ a Cantor set. On the one hand, inflating each set $S_i$ into an open set by taking its Minkowski sum with an open ball of radius $\epsilon$ yields, for any choice of $\epsilon>0$, a family with a connected set of line transversals. Conversely, replacing each $S_i$ by its interior yields a family with an empty set of line transversals. Either change drastically modifies the set of line transversal.

\bibliographystyle{plain} \bibliography{ref}

\end{document}